\documentclass[12pt]{article}

\usepackage[margin=1in]{geometry}
\usepackage{amsmath,amssymb,amsthm,mathtools}
\usepackage{enumitem}
\usepackage[colorlinks=true,linkcolor=blue,citecolor=blue,urlcolor=blue]{hyperref}
\newtheorem{theorem}{Theorem}
\newtheorem{lemma}[theorem]{Lemma}
\newtheorem{proposition}[theorem]{Proposition}

\DeclareMathOperator{\supp}{supp}

\newcommand{\E}{\mathbb E}
\newcommand{\Prob}{\mathbb P}

\newcommand{\R}{\mathbb R}
\newcommand{\one}{\mathbf 1}

\date{}
\title{Return Probability for the Switch--Walk--Switch Lamplighter Walk on a Regular Tree}

\author{ \large Chenyang An\thanks{UCSD. Email: cya.portfolio@gmail.com.} \ \ \ %
Minghao Pan\thanks{Caltech. Email: mpan2@caltech.edu.} }
\begin{document}
\maketitle
\begin{abstract}
We derive the sharp return-probability asymptotic for the switch--walk--switch lamplighter walk with lamp group \(\mathbb Z_2\) over the infinite \(d\)-regular tree:
\[
        p_{2n}(e,e)
        =
        \rho_d^{2n}
        \exp\left[
        -\left(\pi^2(\log(d-1))^2+o(1)\right)
        \frac{n}{\log^2 n}
        \right].
\]
The proofs were generated by QED, a multi-agent system co-developed by the authors, without human intervention beyond the specification of the problem. This provides a test case for the ability of AI systems to produce rigorous mathematical proofs.
\end{abstract}

\section{Introduction}

Let \(T_d\) be the infinite \(d\)-regular tree, where \(d\ge 3\), and
consider the switch--walk--switch lamplighter walk with lamps in
\(\mathbb Z_2\) over \(T_d\). Starting from the all-off configuration with
the lamplighter at a root \(o\), one step of the walk consists of resampling
the lamp at the current vertex, moving the lamplighter to a uniformly chosen
neighbor, and then resampling the lamp at the new vertex. Let
\[
        p_n(e,e)
\]
denote the probability that the walk returns to the identity state after
\(n\) steps.

The main result of this paper identifies the sharp logarithmic
correction to the exponential decay of the return probabilities.

\begin{theorem}\label{thm:main}
Let \(d\ge 3\), let \(T_d\) be the infinite \(d\)-regular tree, and let
\(p_n(e,e)\) be the return probability of the switch--walk--switch
lamplighter walk over \(T_d\) with lamp group \(\mathbb Z_2\). Then, as
\(n\to\infty\),
\[
        p_{2n}(e,e)
        =
        \rho_d^{2n}
        \exp\left[
        -\left(\pi^2(\log(d-1))^2+o(1)\right)
        \frac{n}{\log^2 n}
        \right],
\]
where
\[
        \rho_d=\frac{2\sqrt{d-1}}{d}
\]
is the spectral radius of simple random walk on $T_d$.
\end{theorem}
The starting point is the standard lamplighter--percolation representation
for switch--walk--switch walks; see
Lehner--Neuhauser--Woess~\cite{LNWW}. Let \(\omega\) be Bernoulli site
percolation on \(T_d\) with parameter \(1/2\), and let \(P_\omega\) be the
simple-random-walk transition operator killed on closed vertices:
\[
        P_\omega(x,y)
        =
        \frac1d\mathbf 1_{\{x\sim y\}}
        \mathbf 1_{\{\omega(x)=1\}}
        \mathbf 1_{\{\omega(y)=1\}} .
\]
With the canonical equivariant trace
\[
        \tau(F)=\mathbb E\langle\delta_o,F_\omega\delta_o\rangle,
\]
the return probability satisfies
\[
        p_m(e,e)=\tau(P_\omega^m),\qquad m\ge 0.
\]
Thus, defining the averaged spectral measure
\[
        \nu_d(B)=\tau\bigl(\mathbf 1_B(P_\omega)\bigr),
\]
we have \(\operatorname{supp}\nu_d\subseteq[-\rho_d,\rho_d]\) and
\[
        p_{2n}(e,e)
        =
        \int_{[-\rho_d,\rho_d]}\lambda^{2n}\,d\nu_d(\lambda).
\]
The problem is therefore reduced to estimating the spectral mass of
\(\nu_d\) near the two edges \(\pm\rho_d\).\\

The lower bound comes from finite traps. Let $b:=d-1$. With probability
\(\exp(-O(b^r))\), the percolation cluster of the root contains a forward
\((b)\)-ary tree of depth \(r\). The killed transition operator on this
finite tree has an eigenvalue
\[
        \rho_d\cos\frac{\pi}{r+2}
\]
whose normalized eigenfunction has only polynomially small mass at the
root. Taking \(r\sim \log_{b} n\) balances the probability cost of opening
the trap with the spectral gain, and gives
\[
        p_{2n}(e,e)
        \ge
        \rho_d^{2n}
        \exp\left[
        -\left(\pi^2(\log b)^2+o(1)\right)
        \frac{n}{\log^2 n}
        \right].
\]

The upper bound proves that no substantially better traps exist. The main
deterministic input is a sparse-ball spectral estimate for subtrees of
\(T_d\): if every ball of radius \(r\) inside a subtree contains fewer than
\(a b^r\) vertices, then the killed transition operator has spectral top
strictly below \(\rho_d(1-\delta)\) when \(r\approx(\pi/\sqrt2)\delta^{-1/2}\).
 The probabilistic input is that
dense percolation balls are rare:
\[
        \mathbb P\bigl(|C_\omega(o)\cap B(o,r)|\ge a b^r\bigr)
        \le
        \exp(-c b^r),
\]
where $C_\omega(o)$ is the cluster of the root. 
A finite-von-Neumann-algebra projection comparison then turns this witness
estimate into a double-exponential spectral-edge bound for \(\nu_d\). 
Optimizing the resulting spectral integral at
\(\delta\asymp(\log n)^{-2}\) gives the matching upper bound.

The spectral theory of lamplighter groups has a substantial history. For the
classical lamplighter group over \(\mathbb Z\), spectral measures were
computed by Grigorchuk--\.Zuk~\cite{GrigorchukZuk2001} and further
interpreted by Dicks--Schick~\cite{DicksSchick2002}. Our argument is also related in spirit
to the spectral analysis of random graphs and percolation operators \cite{KirschMueller2006,MuellerStollmann2007,KhorunzhiyKirschMueller2006}.

\subsection*{Cost}
The total AI cost for this work was approximately 1000 US dollars, all spent on API calls. The total AI runtime for these problems was less than 24 hours.

\subsection*{Statement of AI use}
The authors formulated the mathematical problem studied in this paper. The proof was independently generated by QED without human guidance. QED is a multi-agent AI system developed by the
authors together with Qihao Ye and Jiayun Zhang~\cite{QED}. QED decomposes a
problem into intermediate estimates, generates candidate arguments, and
performs automated critique and revision. The authors independently checked the proof and revised the exposition. The authors take full
responsibility for the content of this paper.

\section{Setup and notation}

Throughout the paper \(d\ge 3\) is fixed. Let \(T_d=(V,E)\) be the infinite
\(d\)-regular tree, and fix a root \(o\in V\). We write
\[
        b=d-1,
        \qquad
        \rho_d=\frac{2\sqrt b}{d}.
\]
For \(x\in V\) and \(r\ge 0\), let
\[
        B(x,r)=\{y\in V:d_{T_d}(x,y)\le r\},
        \qquad
        \partial B(x,r)=\{y\in V:d_{T_d}(x,y)=r\}.
\]

Let
\[
        \mathcal L_d=V\times \mathbb Z_2^{(V)}
\]
be the lamplighter state space over \(T_d\), where
\(\mathbb Z_2^{(V)}\) denotes the set of finitely supported lamp
configurations. We write an element as \((x,\eta)\), where
 \(x\in V\) is the lamplighter position and \(\eta:V\to\mathbb Z_2\) is finitely supported. Let 
\[
        e=(o,\mathbf 0)
\]
denote the identity state. 

The switch--walk--switch transition on $\mathcal L_d$ is defined as follows. From
\((x,\eta)\), first resample the lamp at \(x\) from the uniform measure on
\(\mathbb Z_2\), then move from \(x\) to a uniformly chosen neighbor \(y\),
and finally resample the lamp at \(y\) from the uniform measure on
\(\mathbb Z_2\), independently of all previous choices. Let
\[
        p_n(e,e)
\]
be the \(n\)-step return probability to \(e\).

\subsection{Percolation and the killed transition operator}

Let
\[
        \Omega=\{0,1\}^V
\]
with product Bernoulli measure of parameter \(1/2\). A vertex \(x\) is called
open in \(\omega\) if \(\omega(x)=1\). For a configuration \(\omega\), define
the killed transition operator \(P_\omega\) on \(\ell^2(V)\) by
\[
        P_\omega(x,y)
        =
        \frac1d
        \mathbf 1_{\{x\sim y\}}
        \mathbf 1_{\{\omega(x)=1\}}
        \mathbf 1_{\{\omega(y)=1\}} .
\]
Equivalently, \(P_\omega\) is the simple-random-walk transition operator on
the open subgraph, extended by zero on closed vertices. Since \(P_\omega\) is a compression of \(d^{-1}A_{T_d}\),
where \(A_{T_d}\) is the adjacency operator of \(T_d\), and
\[
        \|A_{T_d}\|=2\sqrt b,
\]
we have
\[
        \|P_\omega\|\le \rho_d.
\]

For \(x\in V\), let \(C_\omega(x)\) denote the open cluster of \(x\), with the
convention that \(C_\omega(x)=\varnothing\) if \(x\) is closed.

\subsection{Equivariant random operators and trace}

Let \(\operatorname{Aut}(T_d)\) be the automorphism group of \(T_d\). For
\(g\in\operatorname{Aut}(T_d)\), let \(g\omega\) be the shifted percolation
configuration and let \(U_g\) be the unitary operator on \(\ell^2(V)\)
defined by
\[
        U_g\delta_x=\delta_{gx}.
\]
A bounded random operator \(F=(F_\omega)_\omega\) on \(\ell^2(V)\) is called
equivariant if
\[
        F_{g\omega}=U_gF_\omega U_g^{-1}
        \qquad
        \text{for every }g\in\operatorname{Aut}(T_d).
\]
Let \(\mathcal N\) denote the von Neumann algebra of essentially bounded
equivariant random operators, modulo almost-sure equality. It carries the
canonical trace
\[
        \tau(F)
        =
        \mathbb E_\omega\langle \delta_o,F_\omega\delta_o\rangle .
\]
This trace is faithful, normal, and tracial.

We shall use the standard lattice notation for projections in
\(\mathcal N\). For projections \(p,q\in\mathcal N\), write \(p\le q\) if
\(pq=p\). The meet \(p\wedge q\) is the largest projection dominated by both
\(p\) and \(q\), and the join \(p\vee q\) is the smallest projection
dominating both. The following lemma will be useful in the latter proof. 
\begin{lemma}\label{lm: von Neumann}
 If \(p,q\in \mathcal N\) are projections and
\[
        p\wedge (1-q)=0,
\]
then
\[
        \tau(p)\le \tau(q).
\]
\end{lemma}
\begin{proof}
Using the standard dimension formula for projections in a finite von
Neumann algebra,
\[
        \tau(e\vee f)+\tau(e\wedge f)=\tau(e)+\tau(f),
\]
with \(e=p\) and \(f=1-q\), we get
\[
        \tau(p)+\tau(1-q)
        =
        \tau(p\vee(1-q))+\tau(p\wedge(1-q)).
\]
Since \(p\wedge(1-q)=0\),
\[
        \tau(p)+\tau(1-q)
        =
        \tau(p\vee(1-q))
        \le
        \tau(1)=1.
\]
Therefore
\[
        \tau(p)\le 1-\tau(1-q)=\tau(q).
\]
\end{proof}
\subsection{The averaged spectral measure}

For a Borel set \(B\subseteq\mathbb R\), let
\[
        \mathbf 1_B(P_\omega)
\]
denote the spectral projection of the self-adjoint operator \(P_\omega\)
associated with \(B\). Define a measure \(\nu_d\) by
\begin{equation}\label{eq: def of nu}
            \nu_d(B)
        =
        \tau\bigl(\mathbf 1_B(P_\omega)\bigr)
        =
        \mathbb E\langle \delta_o,\mathbf 1_B(P_\omega)\delta_o\rangle .
\end{equation}
For each $\omega$, the map
\[
        B\longmapsto \langle\delta_o,\one_B(P_\omega)\delta_o\rangle
\]
 is a probability measure. Thus $\nu_d$, being its expectation, is also a probability measure. Since \(\|P_\omega\|\le \rho_d\), the measure \(\nu_d\) is supported on
\([-\rho_d,\rho_d]\).

The lamplighter--percolation identity gives, for every \(m\ge 0\),
\[
        p_m(e,e)
        =
        \tau(P_\omega^m).
\]
Consequently, by the spectral theorem,
\begin{equation}\label{eq: return_spectrum}
            p_{2n}(e,e)
        =
        \tau(P_\omega^{2n})
        =
        \int_{[-\rho_d,\rho_d]}\lambda^{2n}\,d\nu_d(\lambda).
\end{equation}
This equation is the starting point for both the lower and upper
bounds.

\section{The lower bound}

\begin{lemma}\label{lem:trap}
For every integer $r\ge1$ and every $n\ge1$,
\[
        p_{2n}(e,e)
        \ge
        8(r+2)^{-3}
        2^{- 3b^{r+1}}
        \left(\rho_d\cos\frac{\pi}{r+2}\right)^{2n}.
\]

\end{lemma}

\begin{proof}
Fix one neighbor $o^-$ of $o$. Let $B_r$ be the forward rooted $b$-ary tree of depth $r$: it consists of $o$ and all vertices reachable from $o$ by a path of length at most $r$ whose first step is not $o^-$. 

Let $E_r$ be the event that every vertex of $B_r$ is open and every vertex of its external boundary $\partial B_r$ is closed. On $E_r$, the open cluster of $o$ is exactly $B_r$. We have
\[
\begin{aligned}
        |B_r|+|\partial B_r|
        &=
        \frac{b^{r+1}-1}{b-1}+1+b^{r+1}  
        \le
        3b^{r+1}.
\end{aligned}
\]
Then
\begin{equation}\label{eq:8}
       \Prob(E_r)
        =
        2^{-(|B_r|+|\partial B_r|)}
        \ge
        2^{- 3b^{r+1}}.    
\end{equation}

Put
\[
        \theta=\frac{\pi}{r+2}.
\]
Define a radial vector $u$ on $B_r$ by assigning to each level-$j$ vertex the value
\[
        u_j=b^{-j/2}\sin((j+1)\theta),
        \qquad 0\le j\le r.
\]
Let $A_{B_r}$ be the adjacency matrix of $B_r$. We verify that $u$ is an eigenvector. At the root,
\[
        b u_1
        =
        \sqrt b\sin(2\theta)
        =
        2\sqrt b\cos\theta\,u_0.
\]
For $1\le j\le r-1$,
\[
\begin{aligned}
        u_{j-1}+bu_{j+1}
        &=
        b^{-(j-1)/2}\sin(j\theta)
        +
        b\,b^{-(j+1)/2}\sin((j+2)\theta) \\
        &=
        b^{-(j-1)/2}
        \bigl[\sin(j\theta)+\sin((j+2)\theta)\bigr] \\
        &=
        2\sqrt b\cos\theta\,b^{-j/2}\sin((j+1)\theta)\\
        &=
        2\sqrt b\cos\theta\,u_j.
\end{aligned}
\]
At level $r$, the only neighbor inside $B_r$ is the parent. Since $(r+2)\theta=\pi$,
\[
        u_{r-1}
        =
        2\sqrt b\cos\theta\,u_r.
\]
Therefore
\[
        A_{B_r}u=2\sqrt b\cos\theta\,u.
\]
The killed transition operator on $B_r$ is $A_{B_r}/d$, so it has an eigenvalue
\begin{equation}\label{eq: 9}
            \lambda_r
        =
        \frac{2\sqrt b}{d}\cos\theta
        =
        \rho_d\cos\frac{\pi}{r+2}.
\end{equation}

Let $\phi_r$ be the normalized positive eigenfunction. Since level $j$ has $b^j$ vertices,
\[
        \phi_r(o)^2
        =
        \frac{\sin^2\theta}{\sum_{j=0}^r b^j b^{-j}\sin^2((j+1)\theta)}
        =
        \frac{\sin^2\theta}{\sum_{k=1}^{r+1}\sin^2(k\theta)}=\frac{2\sin^2\theta}{r+2}.
\]
Also $\sin\theta\ge 2/(r+2)$ for $r\ge1$. Hence
\begin{equation}
            \phi_r(o)^2
        =
        \frac{2\sin^2\theta}{r+2}
        \ge
        \frac8{(r+2)^3}.
        \label{eq: 10}
\end{equation}

On $E_r$, spectral expansion on the finite cluster $B_r$ gives
\[
        \langle\delta_o,P_\omega^{2n}\delta_o\rangle
        \ge 
        |\langle\delta_o,\phi_r\rangle |^2\lambda_r^{2n}=\phi_r(o)^2\lambda_r^{2n}.
\]
Taking expectation and using (\ref{eq: return_spectrum}),(\ref{eq:8}),(\ref{eq: 9}), and (\ref{eq: 10}), we obtain the asserted bound.\\
\end{proof}

We are now ready to prove the lower bound for the return probability of the switch--walk--switch lamplighter walk, which concludes one direction of Theorem \ref{thm:main}.
\begin{proposition}\label{prop:lower}
As $n\to\infty$,
\[
        p_{2n}(e,e)
        \ge
        \rho_d^{2n}
        \exp\left[-\bigl(\pi^2(\log b)^2+o(1)\bigr)\frac{n}{\log^2 n}\right].
\]
\end{proposition}

\begin{proof}
For all sufficiently large $n$, set
\[
        r_n=\left\lfloor \log_b n-3\log_b\log n\right\rfloor.
\]
Then $r_n\to\infty$, and
\begin{equation}\label{eq: 11}
            b^{r_n+1}
        \le
        b^{\log_b n-3\log_b\log n+1}
        =
        O\left(\frac{n}{(\log n)^3}\right)
        =
        o\left(\frac{n}{\log^2 n}\right).
\end{equation}

Moreover,
\[
        r_n+2
        =
        \frac{\log n}{\log b}(1+o(1)),
        \qquad
        \frac1{(r_n+2)^2}
        =
        \frac{(\log b)^2}{\log^2 n}(1+o(1)).
\]
Since
\[
        \log\cos x=-\frac{x^2}{2}+O(x^4)
        \qquad(x\downarrow0),
\]
we have

\begin{equation} \label{eq: 13}
\begin{split}
    2n\log\cos\frac{\pi}{r_n+2}
        &=
        -\frac{\pi^2 n}{(r_n+2)^2}
        +O\left(\frac{n}{(r_n+2)^4}\right)  \\
        &=
        -\bigl(\pi^2(\log b)^2+o(1)\bigr)
        \frac{n}{\log^2 n}.
\end{split}
\end{equation}

The prefactor in Lemma \ref{lem:trap} contributes to the logarithm only
\begin{equation}
      O(b^{r_n+1})+O(\log r_n)
        =
        o\left(\frac{n}{\log^2 n}\right)   
        \label{eq: 14}
\end{equation}

by (\ref{eq: 11}). Substituting $r=r_n$ in Lemma \ref{lem:trap} and using (\ref{eq: 13}) and (\ref{eq: 14}) gives the claim.
\end{proof}

\section{Two lemmas of spectral analysis}

We need the following lemma, which is the one-dimensional case of Theorem 5.1 of \cite{EGR}.
\begin{lemma}
\label{lem:min-variance-compact-cf}
Let \(\nu\) be a probability measure on \(\mathbb R\), and let
\[
    \widehat \nu(t)
    =
    \int_{\mathbb R} e^{itx}\,d\nu(x)
\]
be its characteristic function. Suppose that
\[
    \widehat \nu(t)=0
    \qquad\text{for all } |t|\ge 1.
\]
Then
\[
    \int_{\mathbb R} x^2\,d\nu(x)\ge \pi^2.
\]
\end{lemma}

\begin{lemma}
\label{lem:small-power-numerical-radius}
Let \(0<\gamma<\pi^2/2\). Then there exist constants
\[
    \varepsilon=\varepsilon(\gamma)>0,
    \qquad
    m_0=m_0(\gamma)<\infty,
\]
such that the following holds. If \(S\) is a contraction on a Hilbert space
and \(m\ge m_0\), then
\[
    \|S^m\|\le \varepsilon
    \quad\Longrightarrow\quad
    w(S)\le 1-\frac{\gamma}{m^2},
\]
where
\[
    w(S)=\sup_{\|u\|=1}|\langle Su,u\rangle|
\]
is the numerical radius of \(S\).
\end{lemma}

\begin{proof}
Suppose the claim fails. Then there exist contractions \(S_j\), integers
\(m_j\to\infty\), and numbers \(\varepsilon_j\downarrow0\) such that
\[
    \|S_j^{m_j}\|\le \varepsilon_j,
    \qquad
    w(S_j)>1-\frac{\gamma}{m_j^2}.
\]
After multiplying \(S_j\) by a scalar of modulus one, we may choose unit
vectors \(\xi_j\) such that
\[
    \operatorname{Re}\langle S_j\xi_j,\xi_j\rangle
   >
    1-\frac{\gamma}{m_j^2}.
\]
By the Sz.-Nagy unitary dilation theorem, there is a unitary \(W_j\) on a
larger Hilbert space, containing \(\xi_j\), such that
\[
    \langle W_j^k\xi_j,\xi_j\rangle
    =
    \langle S_j^k\xi_j,\xi_j\rangle
    \qquad
    \text{for all } k\ge 0.
\]
By the spectral theorem for \(W_j\), there is a probability measure \(\mu_j\)
on the unit circle, written as \(e^{i\theta}\) with \(\theta\in[-\pi,\pi)\),
such that
\[
    \widehat\mu_j(k)
    :=
    \int e^{ik\theta}\,d\mu_j(\theta)
    =
    \langle S_j^k\xi_j,\xi_j\rangle
    \qquad
    \text{for all } k\ge0.
\]

Push \(\mu_j\) forward under the map
\[
    \theta\mapsto m_j\theta,
\]
and call the resulting probability measure on
\([-\pi m_j,\pi m_j)\) by the name \(\nu_j\). From the previous lower bound on
\(\operatorname{Re}\langle S_j\xi_j,\xi_j\rangle\), we get
\[
    \int (1-\cos\theta)\,d\mu_j(\theta)
    \le
    \frac{\gamma}{m_j^2}.
\]
Since
\[
    1-\cos\theta\ge \frac{2\theta^2}{\pi^2}
    \qquad
    \text{for } \theta\in[-\pi,\pi],
\]

The measures \((\nu_j)\) have uniformly bounded second moments. Indeed, since \(\nu_j\) is the pushforward of \(\mu_j\) under
\(t=m_j\theta\),
\[
\begin{aligned}
        \int t^2\,d\nu_j(t)
        &=
        m_j^2\int \theta^2\,d\mu_j(\theta) \\
        &\le
        \frac{\pi^2}{2}m_j^2
        \int(1-\cos\theta)\,d\mu_j(\theta) \\
        &\le
        \frac{\pi^2\gamma}{2}.
\end{aligned}
\]
Thus \((\nu_j)\) has uniformly bounded second moments, and in particular is
tight and has uniformly bounded first moments.

Passing to a subsequence, assume that
\[
    \nu_j\Rightarrow \nu
\]
weakly.

We next show that the characteristic function of \(\nu\) vanishes outside
\((-1,1)\). Fix \(s\ge1\), and put
\[
    k_j=\lfloor s m_j\rfloor.
\]
Then \(k_j\ge m_j\) for all sufficiently large \(j\), and therefore
\[
    \|S_j^{k_j}\|
    \le
    \|S_j^{m_j}\|
    \le
    \varepsilon_j.
\]
Hence
\[
    \left|
    \int e^{i(k_j/m_j)t}\,d\nu_j(t)
    \right|
    =
    |\widehat\mu_j(k_j)|
    =
    |\langle S_j^{k_j}\xi_j,\xi_j\rangle|
    \le
    \varepsilon_j
    \to0.
\]
Since \(k_j/m_j\to s\) and the first moments of \(\nu_j\) are uniformly
bounded, we also have
\[
    \int e^{i(k_j/m_j)t}\,d\nu_j(t)
    -
    \int e^{ist}\,d\nu_j(t)
    \to0.
\]
By weak convergence,
\[
    \int e^{ist}\,d\nu_j(t)
    \to
    \int e^{ist}\,d\nu(t).
\]
Thus
\[
    \widehat\nu(s)=0
    \qquad
    \text{for all } s\ge1.
\]
By conjugation,
\[
    \widehat\nu(s)=0
    \qquad
    \text{for all } s\le -1.
\]
Therefore Lemma~\ref{lem:min-variance-compact-cf} gives
\[
    \int t^2\,d\nu(t)\ge \pi^2.
\]

On the other hand, from
\[
    \int (1-\cos\theta)\,d\mu_j(\theta)
    \le
    \frac{\gamma}{m_j^2},
\]
we get
\[
    \int m_j^2\left(1-\cos\frac{t}{m_j}\right)\,d\nu_j(t)
    \le
    \gamma.
\]
Let \(\chi_R\) be a continuous cutoff satisfying
\[
    0\le \chi_R\le1,
    \qquad
    \chi_R(t)=1 \text{ for } |t|\le R,
    \qquad
    \chi_R(t)=0 \text{ for } |t|\ge R+1.
\]
Since
\[
    m_j^2\left(1-\cos\frac{t}{m_j}\right)
    \to
    \frac{t^2}{2}
\]
uniformly on compact sets, weak convergence gives
\[
    \frac12\int \chi_R(t)t^2\,d\nu(t)
    \le
    \gamma.
\]
Letting \(R\to\infty\), we obtain
\[
    \int t^2\,d\nu(t)\le 2\gamma.
\]
This contradicts
\[
    2\gamma<\pi^2.
\]
The contradiction proves the lemma.
\end{proof}

\section{Sparse-ball spectral estimate }

\begin{proposition}
\label{prop:dense-ball}

For every \(0<\eta<\pi/\sqrt2\), there exist constants
\[
    a_\eta\in(0,1),
    \qquad
    \delta_0=\delta_0(d,\eta)>0,
\]
such that the following deterministic statement holds. For
\(0<\delta<\delta_0\), set
\[
    r_\delta
    =
    \left\lfloor
        \frac{\pi/\sqrt2-\eta}{\sqrt\delta}
    \right\rfloor.
\]
Let \(T\) be a nonempty connected subtree of \(T_d\) and $A_T$ be its adjacency matrix. If
\[
    \sup_{v\in T}|T\cap B_{T_d}(v,r_\delta)|
    <
    a_\eta b^{r_\delta},
\]
then the killed transition operator
\[
    P_T=\frac{1}{d}A_T
\]
satisfies
\[
    \sup\sigma(P_T)
    <
    \rho_d\bigl(1-\delta\bigr).
\]
\end{proposition}

\begin{proof}

Choose $\gamma\in \left((\frac{\pi}{\sqrt2}-\eta)^2,\frac{\pi^2}{2}\right)$.
Let
\[
    \varepsilon=\varepsilon(\gamma)>0,
    \qquad
    m_0=m_0(\gamma)
\]
be given by Lemma~\ref{lem:small-power-numerical-radius}. Define
\[
    a_\eta=\min\{\varepsilon^2,1/2\}.
\]
Choose \(\delta_0>0\) small enough that
\[
    r_\delta\ge m_0
    \qquad
    \text{for all }0<\delta<\delta_0.
\]

Assume that \(T\) satisfies
\[
    \sup_{v\in T}|T\cap B_{T_d}(v,r_\delta)|
    <
    a_\eta b^{r_\delta}.
\]
Since \(a_\eta<1\) and \(r_\delta\ge1\), this hypothesis cannot hold for
\(T=T_d\). Thus \(T\) is a proper connected subtree. By rerooting if necessary, assume that $o\in T$ and $o^-\notin T$ for some neighbor $o^-$ of $o$. Then every vertex of \(T\)
has at most \(b=d-1\) children. Let \(|x|\) denote rooted distance from \(o\),
and define
\[
    m(x)=b^{-|x|}.
\]
The map
\[
    U:\ell^2(T)\to \ell^2(T,m),
    \qquad
    (Uf)(x)=b^{|x|/2}f(x),
\]
is unitary.

Define the rooted shift \(S\) on \(\ell^2(T,m)\) by
\[
    (Sg)(x)
    =
    \begin{cases}
        g(\bar x), & x\ne o,\\
        0, & x=o,
    \end{cases}
\]
where \(\bar x\) is the parent of \(x\). Then
\[
    \|Sg\|_{\ell^2(m)}^2
    =
    \sum_{x\ne o} b^{-|x|}|g(\bar x)|^2
    =
    \sum_{y\in T}\frac{\text{\# of children of $b$ in $T$}}{b}b^{-|y|}|g(y)|^2
    \le
    \|g\|_{\ell^2(m)}^2.
\]
Thus \(S\) is a contraction.

A direct computation gives
\[
    U A_T U^{-1}
    =
    \sqrt b\,(S+S^*).
\]
Indeed, the parent contribution becomes \(\sqrt b\,S\), while the children
contribution becomes \(\sqrt b\,S^*\).

We next estimate \(\|S^{r_\delta}\|\). Let \(D_k(x)\) be the number of
descendants of \(x\) at rooted distance exactly \(k\). Then
\[
    \|S^k\|^2
    =
    \sup_{x\in T}\frac{D_k(x)}{b^k}.
\]
To see this, note that
\[
    \|S^k g\|_{\ell^2(m)}^2
    =
    \sum_{x\in T}
    |g(x)|^2 b^{-|x|}
    \frac{D_k(x)}{b^k}
    \le
    \left(\sup_{x\in T}\frac{D_k(x)}{b^k}\right)
    \|g\|_{\ell^2(m)}^2,
\]
and equality is approached by taking \(g\) supported at a vertex attaining,
or nearly attaining, the supremum.

Since the descendant sphere of radius \(r_\delta\) from \(x\) is contained in
the ordinary ball \(B_{T_d}(x,r_\delta)\), the sparse-ball hypothesis gives
\[
    D_{r_\delta}(x)
    \le
    |T\cap B_{T_d}(x,r_\delta)|
    <
    a_\eta b^{r_\delta}.
\]
Hence
\[
    \|S^{r_\delta}\|^2
    =
    \sup_{x\in T}\frac{D_{r_\delta}(x)}{b^{r_\delta}}
    <
    a_\eta
    \le
    \varepsilon^2.
\]
Therefore
\[
    \|S^{r_\delta}\|\le \varepsilon.
\]
By Lemma~\ref{lem:small-power-numerical-radius},
\[
    w(S)
    \le
    1-\frac{\gamma}{r_\delta^2}.
\]
so that
\[
    w(S)
    <
    1-\delta.
\]

It remains to pass from \(w(S)\) back to the adjacency operator. For
\(\theta\in\mathbb R\), define a unitary \(G_\theta\) on \(\ell^2(T,m)\) by
\[
    (G_\theta g)(x)=e^{i\theta |x|}g(x).
\]
Then
\[
    G_\theta S G_\theta^{-1}=e^{i\theta}S, \qquad \text{and}\qquad      G_\theta S^* G_\theta^{-1}=e^{-i\theta}S^*.
\]
Consequently,
\[
    \|S+S^*\|
    =
    \|e^{i\theta}S+e^{-i\theta}S^*\|
    \qquad
    \text{for every }\theta.
\]
Using the standard numerical-radius identity
\[
    2w(S)
    =
    \sup_{\theta\in\mathbb R}
    \|e^{i\theta}S+e^{-i\theta}S^*\|,
\]
we obtain
\[
    \|S+S^*\|=2w(S).
\]
Therefore
\[
    \sup\sigma(A_T)
    \le
    \|A_T\|
    =
    \sqrt b\,\|S+S^*\|
    =
    2\sqrt b\,w(S)
    <
    2\sqrt b\,\bigl(1-\delta\bigr).
\]
Dividing by \(d\), and recalling that
\[
    \rho_d=\frac{2\sqrt b}{d},
\]
gives
\[
    \sup\sigma(P_T)
    <
    \rho_d\bigl(1-\delta\bigr).
\]
This proves the sparse-ball implication.
\end{proof}
\begin{lemma}\label{lem:dense-rare}
For every $a\in(0,1)$, there are constants $r_0=r_0(d,a)$ and $c=c(d,a)>0$ such that, for every $r\ge r_0$,
\[
        \Prob\left(|C_\omega(o)\cap B(o,r)|\ge a b^r\right)
        \le
        \exp(-c b^r ),
\]
where $C_\omega(o)$ is the open cluster of $o$, interpreted as empty when $o$ is closed.
\end{lemma}

\begin{proof}
Let
\[
        Z_j=|C_\omega(o)\cap \partial B(o,j)|.
\]
We first prove that for every $\alpha>0$ there are constants $c_\alpha>0$ and $j_0$ such that for every $j\geq j_0$
\begin{equation}\label{eq: 25}
            \Prob(Z_j\ge \alpha b^j)
        \le
        \exp(-c_\alpha b^j).
\end{equation}

Let
\[
        D=\frac d b<2.
\]
The sphere $\partial B(o,j)$ has at most $D b^j$ vertices. Choose a finite chain
\[
        \alpha=\alpha_0<\alpha_1<\cdots<\alpha_m<D
\]
such that
\[
        \frac{\alpha_i}{\alpha_{i+1}}>\frac12
        \quad(0\le i<m),
        \qquad
        \alpha_m>\frac D2.
\]
We then choose $j_0\geq m+2 $. 

Fix $i<m$. On the event
\[
        Z_j\ge\alpha_i b^j,
        \qquad
        Z_{j-1}<\alpha_{i+1}b^{j-1},
\]
at least $\alpha_i b^j$ open children must appear among fewer than $\alpha_{i+1}b^j$ possible child sites. Conditional on the previous generation, this requires a binomial random variable with success probability $1/2$ to exceed a fraction at least $\alpha_i/\alpha_{i+1}>1/2$ of its trials. By the Chernoff bound, there is $\gamma_i>0$ such that
\[
        \Prob\bigl(Z_j\ge\alpha_i b^j,
        \ Z_{j-1}<\alpha_{i+1}b^{j-1}\bigr)
        \le
        \exp(-\gamma_i b^j).
\]
Therefore
\begin{equation}\label{eq: 28}
            \Prob(Z_j\ge\alpha_i b^j)
        \le
        \Prob(Z_{j-1}\ge\alpha_{i+1}b^{j-1})
        +
        \exp(-\gamma_i b^j).
\end{equation}

Iterating (\ref{eq: 28}) over $i=0,\ldots,m-1$, it remains to bound
\[
        \Prob(Z_{j-m}\ge\alpha_m b^{j-m}).
\]
But $Z_{j-m}$ is bounded above by the number of open vertices in $\partial B(o,j-m)$, whose size is at most $D b^{j-m}$. Since $\alpha_m>D/2$, another Chernoff bound gives
\[
        \Prob(Z_{j-m}\ge\alpha_m b^{j-m})
        \le
        \exp(-\gamma_m b^{j-m})
        \le
        \exp(-\gamma_m b^{-m}b^j).
\]
Combining this and (\ref{eq: 28}), and decreasing the constant to absorb the fixed number of terms, proves (\ref{eq: 25}).

Now choose an integer $L=L(a,d)$ such that, for all large $r$,
\begin{equation}\label{eq: 30}
 |B(o,r-L)|\le \frac a4 b^r.    
\end{equation}
Put
\[
        \alpha=\frac{a(b-1)}{4b}.
\]
If
\[
        |C_\omega(o)\cap B(o,r)|\ge a b^r,
\]
then there must be some $j\in\{r-L+1,\ldots,r\}$ with
$Z_j\ge\alpha b^j.$
        
Indeed, if all these inequalities failed, then by (\ref{eq: 30}),
\[
\begin{aligned}
        |C_\omega(o)\cap B(o,r)|
        &\le
        |B(o,r-L)|+\alpha\sum_{j=r-L+1}^r b^j \\
        &\le
        \frac a4 b^r
        +
        \frac{a(b-1)}{4b}\frac{b^{r+1}}{b-1}
        =
        \frac a2 b^r,
\end{aligned}
\]
a contradiction.

However, the probability that $ Z_j\ge\alpha b^j$ for some $j\in\{r-L+1,\ldots,r\}$ is at most $\exp(-cb^r)$ by a union bound from $(\ref{eq: 25})$, possibly with a worse constant $c$. This proves the lemma.
\end{proof}

\begin{lemma}
\label{lem:witness-density}
Let \(P_\omega\) be a bounded self-adjoint equivariant random operator on
\(\ell^2(T_d)\). Let \(W_\omega\subseteq V\) be an equivariant random set,
and let \(Q_\omega\) be multiplication by \(\mathbf 1_{W_\omega}\). Fix
\(E\in\mathbb R\). Suppose that
\[
        \mathbf 1_{[E,\infty)}
        \bigl((1-Q_\omega)P_\omega(1-Q_\omega)\bigr)
        =
        0
        \qquad\text{almost surely}.
\]
Then
\[
        \tau\bigl(\mathbf 1_{[E,\infty)}(P_\omega)\bigr)
        \le
        \tau(Q_\omega)
        =
        \mathbb P(o\in W_\omega).
\]
If the same hypothesis holds with \(P_\omega\) replaced by \(-P_\omega\),
then
\[
        \tau\bigl(\mathbf 1_{(-\infty,-E]\cup [E,\infty)}(P_\omega)\bigr)
        \le
        2\mathbb P(o\in W_\omega).
\]
\end{lemma}

\begin{proof}
Set
\[
        \Pi_\omega=\mathbf 1_{[E,\infty)}(P_\omega).
\]
By equivariance and the bounded Borel functional calculus,
\(\Pi=(\Pi_\omega)_\omega\) is a projection in \(\mathcal N\). Similarly,
\(Q=(Q_\omega)_\omega\) is a projection in \(\mathcal N\).

We claim that
\[
        \Pi\wedge(1-Q)=0.
\]
Indeed, suppose
\[
        f\in \operatorname{Ran}\Pi_\omega
        \cap
        \operatorname{Ran}(1-Q_\omega).
\]
Then \(f=(1-Q_\omega)f\). Since \(f\in \operatorname{Ran}\Pi_\omega\), the
spectral theorem gives
\[
        \langle P_\omega f,f\rangle\ge E\|f\|^2.
        \tag{*}
\]
On the other hand, by assumption 
\[
        \langle (1-Q_\omega)P_\omega(1-Q_\omega) f,f\rangle
        =
        \int_{(-\infty,E)} \lambda\,d\mu_f(\lambda)
        <
        E\|f\|^2.
        \tag{**}
\]
Since \(f=(1-Q_\omega)f\),
\[
        \langle P_\omega f,f\rangle
        =
       \langle (1-Q_\omega)P_\omega(1-Q_\omega) f,f\rangle.
\]
Combining \((*)\) and \((**)\) forces \(f=0\). Therefore
\[
        \operatorname{Ran}\Pi_\omega
        \cap
        \operatorname{Ran}(1-Q_\omega)
        =
        \{0\}
        \qquad\text{almost surely},
\]
which is precisely
\[
        \Pi\wedge(1-Q)=0.
\]

By Lemma~\ref{lm: von Neumann},
\[
        \tau(\Pi)\le \tau(Q).
\]
Since \(Q_\omega\) is multiplication by \(\mathbf 1_{W_\omega}\),
\[
        \tau(Q)
        =
        \mathbb E\langle \delta_o,Q_\omega\delta_o\rangle
        =
        \mathbb P(o\in W_\omega).
\]
This proves the positive-edge estimate.

Applying the same argument to \(-P_\omega\) gives
\[
        \tau\bigl(\mathbf 1_{[E,\infty)}(-P_\omega)\bigr)
        \le
        \tau(Q).
\]
Since
\[
        \mathbf 1_{[E,\infty)}(-P_\omega)
        =
        \mathbf 1_{(-\infty,-E]}(P_\omega),
\]
adding the two estimates gives the two-sided bound.
\end{proof}
\section{The upper bound}

\begin{proposition}\label{prop:edge-tail}
Let
\[
        \kappa_d=\frac\pi{\sqrt2}\log b.
\]
For every $\eta>0$, there are constants $c=c(d,\eta)>0$ and $\delta_0=\delta_0(d,\eta)>0$ such that, for all $0<\delta<\delta_0$,
\[
        \nu_d\bigl(\{\lambda:|\lambda|\ge\rho_d(1-\delta)\}\bigr)
        \le
        \exp\left[-c\exp\left(\frac{\kappa_d-\eta}{\sqrt\delta}\right)\right].
\]
\end{proposition}

\begin{proof}
It suffices to prove the estimate for \(0<\eta<\kappa_d\). Apply Proposition \ref{prop:dense-ball} with parameter $\frac{\eta}{4\log b}.$
Let $a\in(0,1)$ be the corresponding density constant. For sufficiently small $\delta$, put
\[
        r_\delta=
        \left\lfloor
        \frac{\pi/\sqrt2-\eta/(4\log b)}{\sqrt\delta}
        \right\rfloor .
\]
Define the witness set
\[
        W_\delta(\omega)
        =
        \left\{
        v:
        |C_\omega(v)\cap B(v,r_\delta)|\ge a b^{r_\delta}
        \right\},
\]
and let $Q_\delta$ be multiplication by $\one_{W_\delta}$.

After deleting $W_\delta$, every remaining open component $T$ satisfies
\[
        \sup_{v\in T}|T\cap B(v,r_\delta)|<a b^{r_\delta}.
\]
After deleting \(W_\delta\), the operator
\((1-Q_\delta)P_\omega(1-Q_\delta)\) is the orthogonal direct sum of the
operators \(P_T\) over the remaining open components \(T\). By Proposition \ref{prop:dense-ball},
each such connected component $T'$ satisfies
\[
        \sup\sigma(P_T')<E,
        \qquad E=\rho_d(1-\delta).
\]
Hence
\[
        1_{[E,\infty)}
        \bigl((1-Q_\delta)P_\omega(1-Q_\delta)\bigr)=0.
\]
Moreover, every open subgraph of $T_d$ is bipartite. Multiplication by $+1$ on one bipartition class and by $-1$ on the other conjugates the adjacency operator to its negative. Hence the same equality holds with $P_\omega$ replaced by $-P_\omega$.

Apply Lemma \ref{lem:witness-density} with
\[
        E=\rho_d(1-\delta).
\]
Using the definition of $\nu_d$ from (\ref{eq: def of nu}), we get
\begin{equation}\label{eq: 37}
            \nu_d\{\lambda:|\lambda|\ge\rho_d(1-\delta)\}
        =
        \tau\bigl(\one_{(-\infty,-E]\cup[E,\infty)}(P_\omega)\bigr) \le
        2\Prob(o\in W_\delta).
\end{equation}
By Lemma \ref{lem:dense-rare},
\begin{equation}\label{eq: 38}
        \Prob(o\in W_\delta)
        \le
        \exp(-c_1 b^{r_\delta})   
\end{equation}
for all sufficiently small $\delta$.

Since
\[
        r_\delta
        =
        \left(\frac\pi{\sqrt2}-\frac{\eta}{4\log b}\right)\delta^{-1/2}+O(1),
\]
we have
\begin{equation}\label{eq: 39}
         \log(b^{r_\delta})
        =
        r_\delta\log b
        \ge
        \frac{\kappa_d-\eta}{\sqrt\delta}
\end{equation}
for all sufficiently small $\delta$. Combining (\ref{eq: 37})--(\ref{eq: 39}), and reducing the constant to absorb the factor $2$, proves the proposition.
\end{proof}


We are now ready to prove Theorem \ref{thm:main}. As the lower bound has been established in Proposition \ref{prop:lower}, it suffices to establish the upper bound. 

\begin{proof}[Proof of Theorem \ref{thm:main}]
Recall
\[
        \kappa_d=\frac\pi{\sqrt2}\log b.
\]
Fix $0<\eta<\kappa_d/4$, and set
\[
        \delta_n=
        \left(\frac{\kappa_d-2\eta}{\log n}\right)^2.
\]
Since $\delta_n\to0$, we assume throughout the rest of the proof that $n$ is large enough so that
\[
        0<\delta_n<\delta_0(d,\eta),
\]
where $\delta_0(d,\eta)$ is the smallness threshold from Proposition \ref{prop:edge-tail}.

Recall from (\ref{eq: return_spectrum}) that
\[
        p_{2n}(e,e)
        =
        \int |\lambda|^{2n}\,d\nu_d(\lambda).
\]
Split the integral into
\[
        I_1=
        \int_{\{|\lambda|\le\rho_d(1-\delta_n)\}}
        |\lambda|^{2n}\,d\nu_d(\lambda),
\]
and
\[
        I_2=
        \int_{\{|\lambda|>\rho_d(1-\delta_n)\}}
        |\lambda|^{2n}\,d\nu_d(\lambda).
\]
For the first term,
\begin{equation}\label{eq: 41}
    \begin{split}
            I_1
        &\le
        \rho_d^{2n}(1-\delta_n)^{2n} \\
        &\le
        \rho_d^{2n}\exp(-2n\delta_n) \\
        &=
        \rho_d^{2n}
        \exp\left[-2(\kappa_d-2\eta)^2\frac{n}{\log^2 n}\right].
    \end{split}
\end{equation}

For the second term, using the support condition $\supp\nu_d\subseteq[-\rho_d,\rho_d]$ and Proposition \ref{prop:edge-tail} with $\delta=\delta_n$, we get
\[
\begin{aligned}
        I_2\le
        \rho_d^{2n}
        \nu_d\{|\lambda|\ge\rho_d(1-\delta_n)\} 
        \le
        \rho_d^{2n}
        \exp\left[-c\exp\left(\frac{\kappa_d-\eta}{\sqrt{\delta_n}}\right)\right].
\end{aligned}
\]
Since
\[
        \sqrt{\delta_n}
        =
        \frac{\kappa_d-2\eta}{\log n},
\]
we have
\[
        \exp\left(\frac{\kappa_d-\eta}{\sqrt{\delta_n}}\right)
        =
        n^{(\kappa_d-\eta)/(\kappa_d-2\eta)}.
\]
The exponent
\[
        \frac{\kappa_d-\eta}{\kappa_d-2\eta}>1,
\]
so $I_2$ is smaller than
\[
        \rho_d^{2n}\exp\left[-M\frac{n}{\log^2 n}\right]
\]
for every fixed $M>0$, once $n$ is sufficiently large.

Combining this with (\ref{eq: 41}),
\[
        p_{2n}(e,e)
        \le
        \rho_d^{2n}
        \exp\left[-\bigl(2(\kappa_d-2\eta)^2-o(1)\bigr)
        \frac{n}{\log^2 n}\right].
\]
Therefore, for every fixed \(0<\eta<\kappa_d\),
\[
        \liminf_{n\to\infty}
        \frac{-\log(p_{2n}(e,e)/\rho_d^{2n})}{n/\log^2 n}
        \ge
        2(\kappa_d-2\eta)^2.
\]
Letting \(\eta\downarrow0\), and using
\[
        2\kappa_d^2
        =
        2\left(\frac{\pi}{\sqrt2}\log b\right)^2
        =
        \pi^2(\log b)^2,
\]
gives
\[
        \liminf_{n\to\infty}
        \frac{-\log(p_{2n}(e,e)/\rho_d^{2n})}{n/\log^2 n}
        \ge
        \pi^2(\log b)^2.
\]
Together with Proposition~\ref{prop:lower}, this proves Theorem~1.

\end{proof}


\begin{thebibliography}{9}

\bibitem{QED}
C. An, Q. Ye, M. Pan, and J. Zhang,
\emph{QED: An open-source multi-agent system for generating mathematical proofs on open problems},
arXiv:2604.24021 (2026).

\bibitem{DicksSchick2002}
W. Dicks and T. Schick,
\emph{The spectral measure of certain elements of the complex group ring of a wreath product},
Geometriae Dedicata \textbf{93} (2002), 121--137.

\bibitem{EGR}
W. Ehm, T. Gneiting, and D. Richards,
\emph{Convolution roots of radial positive definite functions with compact support},
Transactions of the American Mathematical Society \textbf{356} (2004), 4655--4685.

\bibitem{GrigorchukZuk2001}
R. I. Grigorchuk and A. \.Zuk,
\emph{The lamplighter group as a group generated by a 2-state automaton, and its spectrum},
Geometriae Dedicata \textbf{87} (2001), 209--244.

\bibitem{KhorunzhiyKirschMueller2006}
O. Khorunzhiy, W. Kirsch, and P. Müller,
\emph{Lifshitz tails for spectra of Erdős--Rényi random graphs},
The Annals of Applied Probability \textbf{16} (2006), 295--309.

\bibitem{KirschMueller2006}
W. Kirsch and P. Müller,
\emph{Spectral properties of the Laplacian on bond-percolation graphs},
Mathematische Zeitschrift \textbf{252} (2006), 899--916.

\bibitem{LNWW}
F. Lehner, M. Neuhauser, and W. Woess,
\emph{On the spectrum of lamplighter groups and percolation clusters},
Mathematische Annalen \textbf{342} (2008), 69--89.


\bibitem{MuellerStollmann2007}
P. Müller and P. Stollmann,
\emph{Spectral asymptotics of the Laplacian on supercritical bond-percolation graphs},
Journal of Functional Analysis \textbf{252} (2007), 233--246.



\end{thebibliography}
\end{document}